\documentclass{article}
\usepackage{amsmath}
\usepackage{amsfonts}
\usepackage{amsthm}
\usepackage{amssymb}
\usepackage{color}

\newtheorem{theorem}{Theorem}[section]

\newtheorem{problem}[theorem]{Problem}

\theoremstyle{definition}
\newtheorem{definition}[theorem]{Definition}
\theoremstyle{remark}
\newtheorem{remark}[theorem]{Remark}

\def\N{{\mathbf{N}}}

\def\mju{\mathcal{U}}

\def\f2{\mathbb{F}_2}

\newcommand{\ep}{\varepsilon}

\begin{document}

\title{\LARGE There is no finitely isometric Krivine's theorem}

\author{James~Kilbane and Mikhail~I.~Ostrovskii}

\date{\today}
\maketitle

\begin{large}

\noindent{\bf Abstract.} We prove that for every $p\in(1,\infty)$,
$p\ne 2$, there exist a Banach space $X$ isomorphic to $\ell_p$
and a finite subset $U$ in $\ell_p$, such that $U$ is not
isometric to a subset of $X$. This result shows that the finite
isometric version of the Krivine theorem (which would be a
strengthening of the Krivine theorem (1976)) does not hold.
\medskip

{\small \noindent{\bf Keywords:} isometric embedding, isomorphism
of Banach spaces, Krivine theorem, Orlicz space, modular space
\medskip

\noindent{\bf 2010 Mathematics Subject Classification.} Primary:
46B03; Secondary: 30L05, 46B07, 46B85}

\section{Introduction}

One of the most fundamental results on the structure of the
general infinite-dimensio\-nal Banach spaces is the following
theorem of Dvoretzky.

\begin{theorem}[Dvoretzky \cite{Dvo61}]\label{T:Dvo} For each infinite-di\-men\-sional Banach space $X$, each
$n\in\mathbb{N}$, and each $\ep>0$, there is an $n$-dimensional
subspace $X_n\subset X$ and an isomorphism $T:X_n\to\ell_2^n$ such
that $||T||\cdot||T^{-1}||\le 1+\ep$.
\end{theorem}

It is well known that $(1+\ep)$ cannot be replaced by $1$ in this
theorem. This follows, for example, from the fact that the unit
ball of any finite-dimensional subspace in $c_0$ is a polytope.
The fact that $\ell_p$ does not contain all of $\ell_2^n$
isometrically, unless $p$ is an even integer, was proved in
\cite{DJP98}.
\medskip

Recently embeddability of finite metric spaces into Banach spaces
became a very important direction in the Banach space theory. One
of the main reasons for this is the discovery that such embeddings
have important algorithmic applications, see \cite{LLR95,AR98}.
Low-distortion embedding of finite metric spaces into Banach
spaces became a very powerful toolkit for designing efficient
algorithms, the interested reader can find more information in surveys such as \cite{Ind01, IM04, Lin02,
Mat02, Nao10, Ost13, WS11}.\medskip

In this connection it is worthwhile to observe that Theorem
\ref{T:Dvo} can be derived from the following seemingly weaker
theorem. Our terminology
follows \cite{Ost13}.

\begin{theorem}[Finite Dvoretzky Theorem]\label{T:FinDvo}
For each infinite-di\-men\-sional Banach space $X$, each finite
subset $F\subset \ell_2$, and each $\ep>0$, there is a bilipschitz
embedding of $F$ into $X$ with distortion at most $(1+\ep)$.
\end{theorem}

\begin{proof}[Proof of {\rm (Theorem
\ref{T:FinDvo})$\Rightarrow$(Theorem \ref{T:Dvo})}] We are going
to use ultrapowers of Banach spaces (see \cite[Section
2.2]{Ost13}). We need to show that if the conclusion of Theorem
\ref{T:FinDvo} holds for a Banach space $X$, then there exists an
ultrapower of $X$ containing an isometric copy of $\ell_2$. This
can be done as follows (this is a slightly modified version of
\cite[Proposition 2.33]{Ost13}).
\medskip

Denote by $J$ the set of all finite subsets of $\ell_2$ containing
$0$. Consider the set $I=J\times(0,1)$ as an ordered set:
$(j_1,\ep_1)\succeq(j_2,\ep_2)$ if and only if $j_1\supseteq j_2$
and $\ep_1\le \ep_2$. Consider an ultrafilter $\mju$ on $I$
containing the filter generated by sets of the form $\{(j,\ep):~
(j,\ep)\succeq(j_0,\ep_0)\}$, where $j_0\in J$, $\ep_0\in(0,1)$.
\medskip

The conclusion of Theorem \ref{T:FinDvo} implies that  for each
pair $(j,\ep)\in J\times(0,1)$ there is a map $T_{(j,\ep)}:j\to X$
with distortion $\le 1+\ep$ satisfying $T(0)=0$.

It remains to observe that the maps
\[z\mapsto \begin{cases} T_{(j,\ep)}(z) &~\hbox{ if } z\in j\\
0 &~\hbox{ if } z\notin j
\end{cases}\]
(parameterized by pairs $(j,\ep)\in I$)  induce an isometric
embedding of $\ell_2$ into $X^\mju$.
\end{proof}

An important difference between Theorem \ref{T:Dvo} and Theorem
\ref{T:FinDvo} is that there are no known examples showing the
necessity of the $+\,\ep$ in Theorem \ref{T:FinDvo}. The examples
above showing the necessity of the $+\,\ep$ in Theorem \ref{T:Dvo} do
not serve as examples in the finite case. In fact, Fr\'echet
\cite{Fre10} (see also \cite[Proposition~1.17]{Ost13}) proved that
each $n$-element set embeds isometrically into $\ell_\infty^n$,
and thus, into $c_0$. Ball \cite{Bal90} proved that each $n$-element
subset of $L_p$ embeds isometrically into $\ell_p^{\binom{n}{2}}$. Since, as is well known \cite[p.~16]{JL01}
$\ell_2$ embeds isometrically into $L_p[0,1]$ for every $p$, it
follows that for $X=\ell_p$ the statement of  finite Dvoretzky
theorem remains true if we replace $(1+\ep)$ by $1$.

In this connection, the second-named author asked whether the
result which can be called ``finite isometric Dvoretzky theorem''
is true for all infinite-dimensional Banach spaces $X$, that is,

\begin{problem}[\cite{Ost15a}] Does there exist a finite subset $F$ of $\ell_2$
and an infinite-dimensional Banach space $X$ such that $F$ does
not admit an isometric embedding into $X$?
\end{problem}

This problem remains open. In this paper we show that the result
which could be called ``finite isometric Krivine theorem'' does
not hold for any $p\in[1,\infty]$, $p\ne 2$. More precisely, we
answer in the negative, for every $p\in(1,\infty)$, $p\ne 2$, the
following problem suggested in \cite{Ost15b}:

\begin{problem}[\cite{Ost15b}]\label{P:FinIsomKriv} Let $Y$ be a Banach space isomorphic to $\ell_p$, $1<p<\infty$.
Is it true that any finite subset of $\ell_p$ is isometric to some finite subset of $Y$?
\end{problem}

To justify the term ``finite isometric Krivine theorem'' let us
recall the following landmark result of Krivine.

\begin{theorem}[Krivine \cite{Kri76}] For each $p\in[1,\infty]$, each Banach space $X$ isomorphic to $\ell_p$, each
$n\in\mathbb{N}$, and each $\ep>0$, there is an $n$-dimensional
subspace $X_n\subset X$ and an isomorphism $T:X_n\to\ell_p^n$ such
that $||T||\cdot||T^{-1}||\le 1+\ep$.
\end{theorem}

The cases $p=1$ and $p=\infty$ were not included in Problem
\ref{P:FinIsomKriv} because Bill~Johnson had already described
examples in these cases in his answer to \cite{Kil15}. The
examples are the following: both $\ell_1$ and $\ell_\infty$ are
isomorphic to strictly convex spaces. On the other hand, both
$\ell_1$ and $\ell_\infty$ contain quadruples of points $a,b,c,d$
such that $b\ne c$ and both $b$ and $c$ are metric midpoints
between $a$ and $d$. It is easy to see that such quadruples do not
exist in strictly convex Banach spaces.

It is worth mentioning that although we prove that the answer to
Problem \ref{P:FinIsomKriv} is negative, there exist ``many''
subsets of the unit sphere of $\ell_p$ for which the result is
positive, see the paper \cite{Kil17+} of the first-named author
for precise statement.\medskip

Our main result is the following (we denote by $\{e_i\}$ the unit
vector basis of $\ell_p$):

\begin{theorem}\label{T:NoFinIsKriv} {\rm (a)} For each $1<p<2$ there exist a Banach
space $X$ isomorphic to $\ell_p$ such that the set
$U:=\{e_1,e_2,-e_1,-e_2,0\}$, considered as a subset of $\ell_p$,
does not embed isometrically into $X$.
\medskip

\noindent{\rm (b)}  For each $2<p<\infty$ there exist a Banach
space $X$ isomorphic to $\ell_p$ such that the set
$V:=\{\pm2^{-1/p}(e_1+e_2),\pm2^{-1/p}(e_1-e_2),0\}$, considered
as a subset of $\ell_p$, does not embed isometrically into $X$.

\end{theorem}

\begin{remark} The coefficient $2^{-1/p}$ in the statement of (b)
is needed to make the vectors normalized (this will be
convenient), of course the result holds without this coefficient.
\end{remark}

The main technical tools we will use in the proof of Theorem
\ref{T:NoFinIsKriv} are the Clarkson inequalities. In the
following theorem, if $q \in (1,\infty)$ we set $q^\prime$ to be
the so-called conjugate index of $q$, defined by $\frac1q +
\frac1{q^\prime} = 1$. We recall the following (see \cite[Theorem
9.7.2]{Gar07} for generalized Clarkson inequalities):
\begin{theorem}[{\cite[Theorem 2]{Cla36}}]\label{T:GenClark}
Suppose that $x,y \in \ell_p$, where $1 < p < \infty$, and
$r=\min(p,p^\prime)$. Then,
\begin{itemize}
    \item[{\rm(1)}] $2(\|x\|_p^{r^\prime} + \|y\|_p^{r^\prime}) \leq \|x+y\|_p^{r^\prime} + \|x-y\|_p^{r^\prime} \leq 2^{r^\prime - 1} (\|x\|_p^{r^\prime} + \|y\|_p^{r^\prime})$
    \item[{\rm(2)}]  $2^{r-1} (\|x\|_p^r + \|y\|_p^r) \leq \|x+y\|_p^r + \|x-y\|_p^r \leq 2(\|x\|_p^r + \|y\|_p^r)$.
\end{itemize}
\end{theorem}

\begin{remark} The following remark is for the reader who knows the definition of the \emph{James constant} of a Banach space, which is defined as $J(X) = \sup \{ \min ( \|x+y\|, \|x-y\| ) : x,y \in S_X\}$. One can unify parts (a) and (b) of Theorem \ref{T:NoFinIsKriv} in terms of the following:
\end{remark}
\begin{theorem}\label{T:Jam}
For each $p \neq 2$ there is a Banach space $X$ such that $X$ is isomorphic to $\ell_p$, $J(X) = J(\ell_p)$, but the supremum in the definition of the James constant is not attained.
\end{theorem}
Proving Theorem \ref{T:Jam} one can use some of the results of \cite{RR02} and \cite{Yan07}. To make our argument as elementary and self-contained we prefer to present a direct argument in terms of the metric spaces $U$ and $V$.

\section{The case $p\in(1,2)$}\label{S:p<2}

We show that in this case we can choose $X$ to be an Orlicz sequence space $\ell_M$ for a suitably chosen function $M(t)$. Let us recall the definition of an Orlicz sequence space.

\begin{definition}
Let $M:[0,\infty) \rightarrow [0,\infty)$ be a continuous, non-decreasing and convex function such that $M(0) = 0$ and $\lim_{t \rightarrow \infty} M(t) = \infty$. We define the sequence space $\ell_M$ to be the collection of sequences $x = (x_1,x_2,\dots)$ such that $\sum M(|x_n| / \rho) < \infty$ for some $\rho$ and define the norm $\|x\|_M$ to be \[ \|x\|_M = \inf \left\{ \rho > 0 : \sum_{i=1}^\infty M\left(\frac{|x_i|}\rho\right) \leq 1 \right\} \]
\end{definition}

We refer to \cite[Section 4.a]{LT77} for basic properties of Orlicz sequence spaces, however, our proof will require very little of this theory to understand.

Let $p\in(1,2)$, pick any $r\in(p,2)$ and let $M(t) = t^p +
t^r$. We show that the corresponding Orlicz space $\ell_M$ has all
of the desired properties. The fact that $\ell_M$ is isomorphic to
$\ell_p$ follows immediately from \cite[Proposition~4.a.5]{LT77}.

Assume that $\ell_M$ does not have the second property, that is,
assume that $U$ admits an isometric embedding $f$ into $\ell_M$.
Without loss of generality we may assume that $f(0)=0$. Denote
$f(e_1)$ by $x$ and $f(e_2)$ by $y$. It is easy to see that, since
$\ell_M$ is a strictly convex space \cite[Chapter VII]{RR91}, we have $f(-e_1)=-x$ and $f(-e_2)=-y$. We have $||x||=||y||=1$. So we need to get a contradiction by showing that
it is not possible that both of the vectors: $x+y$ and $x-y$ have
norm $2^{\frac1p}$ in $\ell_M$.

Since $x,y \in S_{\ell_M}$, we have $\sum_{i=1}^\infty |x_i|^p +
\sum_{i=1}^\infty |x_i|^r = 1$ and  $\sum_{i=1}^\infty |y_i|^p +
\sum_{i=1}^\infty |y_i|^r = 1$. We can write this as $\|x\|_p^p +
\|x\|_r^r = 1$ and $\|y\|_p^p + \|y\|_r^r = 1$, where by
$\|\cdot\|_p$ we denote the norm of a sequence in $\ell_p$. Adding
these equalities we get

\[2 =  \|x\|_p^p + \|y\|_p^p + \|x\|_r^r
+ \|y\|_r^r.\] By the Clarkson inequality ((2) of Theorem
\ref{T:GenClark}), we get

\begin{equation}\label{E:FromAbove} 4 \geq \|x+y\|_p^p +  \|x-y\|_p^p+
\|x+y\|_r^r + \|x-y\|_r^r \end{equation}

Denote $u = \|x+y\|_M$ and $v = \|x-y\|_M$, and note that
$$ \frac{\|x+y\|_p^p}{u^p} + \frac{\|x+y\|_r^r}{u^r} =
1$$ and
$$ \frac{\|x-y\|_p^p}{v^p} +
\frac{\|x-y\|_r^r}{v^r} = 1.$$

Suppose that $u=v= 2^{1/p}$, i.e., the embedding described above is
isometric. We get that
$$\frac{1}{2}\, \|x+y\|_p^p + \frac{1}{2^{r/p}} \|x+y\|_r^r =
1$$ and
$$\frac{1}{2}\, \|x-y\|_p^p + \frac{1}{2^{r/p}} \|x-y\|_r^r =1.$$

Doubling and adding gives
\[\|x+y\|_p^p +  \|x-y\|_p^p+
2^{1-\frac{r}p}\|x+y\|_r^r + 2^{1-\frac{r}p}\|x-y\|_r^r=4\] Since
$2^{1-r/p} < 1$, we get a contradiction with \eqref{E:FromAbove}.

\section{Case $p\in(2,\infty)$}

This case is more difficult. The reason is the following: to get a
counterexample in the above we ``perturbed'' $\ell_p$ slightly ``in the
direction of $\ell_2$''. While for $p<2$ this is achievable by
addition of $t^r$ to the Orlicz function corresponding to
$\ell_p$, this is no longer possible for $p>2$ (the space
corresponding to $t^p+t^r$ with $r<p$ is isomorphic to $\ell_r$,
and not to $\ell_p$), for this reason we have to consider more
complicated, so-called {\it modular spaces}.

Let us recall the definition of modular spaces:
\begin{definition} \label{D:modular}For each $i \in \N$, let $M_i: [0,\infty) \rightarrow [0,\infty)$ be a continuous, convex, non-decreasing and convex function such that $M(0) = 0$ and $\lim_{t \rightarrow \infty} M(t) = \infty$. Then the {\it modular sequence space} $\ell_{\{M_i\}}$ is the Banach space
of all sequences $x=\{x_i\}_{i=1}^\infty$ with $\sum_{i=1}^\infty
M_i(|x_i|/\rho)<\infty$ for some $\rho>0$, equipped with the norm

\[||x||=\inf\left\{\rho>0:\quad \sum_{i=1}^\infty
M_i \left( \frac{|x_i|}\rho \right)\le 1\right\}.\]
\end{definition}

We refer to \cite[Section~4.d]{LT77} for basic information on modular sequence spaces.
\medskip

Let $p>2$. We introduce a sequence $\{M_i\}_{i=1}^\infty$ of
functions given by $M_i(t) = t^p + t^{p_i}$ where $2 < p_i < p$ for
any $i$, and the sequence $p_i$ converges to $p$ rapidly enough,
so that the obtained modular space $\ell_{\{M_i\}}$ is isomorphic
to $\ell_p$. To see that this is achievable, we recall the following criterion from \cite[p.~167]{LT77}: if $M_i$ and $N_i$ are two collections of functions having the properties in Definition \ref{D:modular}, then $\ell_{\{M_i\}}$ and $\ell_{\{N_i\}}$ are isomorphic with the
identity map being an isomorphism if there
exist numbers $K>0$, $t_i\ge 0$, $i=1,2,\dots$, and an integer
$i_0$ so that

\begin{itemize}

\item[{\rm (a)}] $K^{-1}N_i(t)\le M_i(t)\le KN_i(t)$ for all $i\ge
i_0$ and $t\ge t_i$.

\item[{\rm (b)}] $\sum_{i=1}^\infty N_i(t_i)<\infty$.

\end{itemize}

We are going to apply this criterion with $N_i(t)=t^p$ for every
$i\in\mathbb{N}$. We choose $t_i>0$ to be a convergent to $0$
sequence for which $\sum_{i=1}^\infty t_i^p<\infty$ (so (b) is
satisfied). Finally, we let $i_0=1$ and choose the sequence
$\{p_i\}\in(2,p)$ so rapidly approaching $p$, that the condition
(a) is satisfied with $K=3$. It is easy to see that this is
possible.

We wish to use an argument similar to the argument in Section \ref{S:p<2}. To do this we will need to first prove the strict convexity of $\ell_{\{M_i\}}$. To show this, suppose that we pick two distinct elements $u,v \in S_{\ell_{\{M_i\}}}$. This means that
\begin{equation}\label{E:x,p>2}\sum_{i=1}^\infty( |u_i|^p + |u_i|^{p_i}) =
||u||_p^p+\sum_{i=1}^\infty|u_i|^{p_i} =1\end{equation} and
\begin{equation}\label{E:y,p>2}\sum_{i=1}^\infty( |v_i|^p + |v_i|^{p_i}) =
||v||_p^p+\sum_{i=1}^\infty|v_i|^{p_i} =1\end{equation}

Adding these together, and using the Clarkson inequality ((1) of
Theorem \ref{T:GenClark}), we get
\begin{equation}\label{E:Above,p>2} 2^{1-p} (\|u+v\|_p^p +
\|u-v\|_p^p) + \sum_{i=1}^\infty  2^{1-p_i}
(|u_i+v_i|^{p_i}+|u_i-v_i|^{p_i})  \leq 2\end{equation}

To show that $\|u+v\| < 2$, assume that $\|u+v\| = 2$, that is, \[ 2^{-p} \|u+v\|_p^p + \sum_{i=1}^\infty 2^{-p_i} |u_i+v_i|^{p_i} = 1. \]Multiplying by 2 and comparing with Equation \eqref{E:Above,p>2} is a contradiction, therefore $\|u+v\| < 2$.

We now continue, and show that $V$ does not isometrically embed into $\ell_{\{M_i\}}$.
Observe that the distance in $\ell_p$ between any of the vector
$\pm2^{-1/p}(e_1+e_2)$ and any of the vectors $\pm
2^{-1/p}(e_1-e_2)$ is equal to $2^{1-\frac1p}$. Assume that $V$
admits an isometric embedding $f$ into $\ell_{\{M_i\}}$. Without
loss of generality we may assume that $f(0)=0$. Set
$x=f(2^{-1/p}(e_1+e_2))$ and $y=f(2^{-1/p}(e_1-e_2))$. By the
strict convexity we get $-x=f(-2^{-1/p}(e_1+e_2))$ and
$-y=f(-2^{-1/p}(e_1-e_2))$. We have
$||x||_{\{M_i\}}=||y||_{\{M_i\}}=1$. To complete the proof it
suffices to show that
\[||x-y||_{\{M_i\}}=||x+y||_{\{M_i\}}=2^{1-\frac1p}\]
leads to a contradiction. This gives us that
\[ \frac{\|x+y\|_p^p}{2^{p-1}} +
\sum_{i=1}^\infty \frac{|x+y|^{p_i}}{2^{p_i(1-\frac1p)}} = 1
\]
and
\[ \frac{\|x-y\|_p^p}{2^{p-1}} +
\sum_{i=1}^\infty \frac{|x-y|^{p_i}}{2^{p_i(1-\frac1p)}} = 1
\]
Adding and rearranging we get

\begin{equation}\label{E:Equal,p>2} 2^{1-p} (\|x+y\|_p^p + \|x-y\|_p^p) + \sum_{i=1}^\infty
2^{p_i(\frac1p-1)} (|x_i+y_i|^{p_i}+|x_i-y_i|^{p_i}) =
2\end{equation}

Since $2<p_i<p$, and therefore $1-p_i>{p_i(\frac1p-1)}$, the
equations \eqref{E:Above,p>2} (which was valid for any elements of $S_{\ell_{\{M_i\}}}$, so we are free to set $u = x$ and $v = y$) and \eqref{E:Equal,p>2} contradict
each other.

\section{Acknowledgement}

The second-named author gratefully acknowledges the support by
National Science Foundation DMS-1700176.

\end{large}

\begin{small}

\renewcommand{\refname}{
\end{small}

\textsc{(J.K.) Department of Pure Maths and Mathematical
Statistics, University of Cambridge, Cambridge, CB3 0WB, UK}\par

\textit{E-mail address}: \texttt{jk511@cam.ac.uk}
\bigskip

\textsc{(M.O.) Department of Mathematics and Computer Science, St.
John's University, 8000 Utopia Parkway, Queens, NY 11439, USA}
\par

\textit{E-mail address}: \texttt{ostrovsm@stjohns.edu} \par
  \smallskip

\end{document}